\documentclass[11pt]{article}
\usepackage[letterpaper,hmargin=1in,vmargin=1.25in]{geometry}

\usepackage{url}

\usepackage{amsmath}
\usepackage{amssymb}

\usepackage{amsthm}

\theoremstyle{plain}
\newtheorem{theorem}{Theorem}[section]
\newtheorem{lemma}[theorem]{Lemma}
\newtheorem{corollary}[theorem]{Corollary}

\newtheorem{definition}[theorem]{Definition}
\newtheorem{remark}[theorem]{Remark}

\newtheorem{fact}{Fact}
\newtheorem{problem}{Problem}

\newcommand{\abs}[1]{\left\lvert#1\right\rvert}

\DeclareMathOperator{\Dom}{dom}

\DeclareMathOperator{\bin}{bin}
\newcommand{\rest}[2]{#1\!\!\restriction_{#2}}
\newcommand{\reste}[2]{#1\restriction_{#2}}

\newcommand{\N}{\mathbb{N}}
\newcommand{\Z}{\mathbb{Z}}
\newcommand{\Q}{\mathbb{Q}}
\newcommand{\R}{\mathbb{R}}
\newcommand{\X}{\{0,1\}^*}
\newcommand{\XI}{\{0,1\}^\infty}

\newcommand{\noi}{\noindent}

\begin{document}


\begin{center}
{\Large \textbf{\boldmath
  Chaitin $\Omega$ numbers and halting problems
}}
\end{center}

\vspace{-2.2mm}

\begin{center}
Kohtaro Tadaki
\end{center}

\vspace{-5mm}

\begin{center}
Research and Development Initiative, Chuo University\\
CREST, JST\\
1--13--27 Kasuga, Bunkyo-ku, Tokyo 112-8551, Japan\\
E-mail: tadaki@kc.chuo-u.ac.jp\\
http://www2.odn.ne.jp/tadaki/
\end{center}

\vspace{-2mm}

\begin{quotation}
\noi\textbf{Abstract.}
Chaitin
[G. J. Chaitin,
\emph{J. Assoc. Comput. Mach.}, vol.\ 22, pp.~329--340, 1975]
introduced $\Omega$ number
as a concrete example of random real.
The real $\Omega$ is defined
as the probability that an optimal computer halts,
where the optimal computer is
a universal decoding algorithm used
to define the notion of program-size complexity.
Chaitin showed $\Omega$ to be random
by discovering the property that
the first $n$ bits of the base-two expansion of $\Omega$ solve
the halting problem of the optimal computer
for all binary inputs of length at most $n$.
In the present paper
we investigate this property from various aspects.
We consider the relative computational power
between the base-two expansion of $\Omega$ and
the halting problem
by imposing the restriction to finite size
on both the problems.
It is known that
the base-two expansion of $\Omega$ and
the halting problem
are Turing equivalent.
We thus consider an elaboration of the Turing equivalence
in a certain manner.
\end{quotation}

\begin{quotation}
\noi\textit{Key words\/}:
algorithmic information theory,
Chaitin $\Omega$ number,
halting problem,
Turing equivalence,
algorithmic randomness,
program-size complexity
\end{quotation}

\section{Introduction}

Algorithmic information theory (AIT, for short) is a framework
for applying
information-theoretic and probabilistic ideas to recursive function theory.
One of the primary concepts of AIT is the \textit{program-size complexity}
(or \textit{Kolmogorov complexity}) $H(s)$ of a finite binary string $s$,
which is defined as the length of the shortest binary
input
for a universal decoding algorithm $U$,
called an \textit{optimal computer},
to output $s$.
By the definition,
$H(s)$
can be thought of as
the information content of the individual finite binary string $s$.
In fact,
AIT has precisely the formal properties of
classical information theory (see Chaitin \cite{C75}).
In particular,
the notion of program-size complexity plays a crucial role in
characterizing the \textit{randomness} of an infinite binary string,
or equivalently, a real.
In \cite{C75} Chaitin introduced the halting probability $\Omega_U$
as an example of
random real.
His $\Omega_U$ is defined
as the probability that the optimal computer $U$ halts,
and plays a central role in the
metamathematical
development of
AIT.
The real $\Omega_U$ is shown to be random,
based on the following fact:%
\footnote{
A rigorous form of Fact~\ref{Chaitin} is
seen in Theorem~\ref{OmegaVhaltC} below
in a more general form.
}

\begin{fact}[Chaitin \cite{C75}]\label{Chaitin}
The first $n$ bits of the base-two expansion of $\Omega_U$ solve
the halting problem of $U$ for inputs of
length at most $n$.
\qed
\end{fact}

In this paper, we first consider the
following
converse problem:

\begin{problem}\label{motivation}
For every positive integer $n$,
if $n$ and the list of all halting inputs for $U$
of length at most $n$ are given,
can the first $n$ bits of the base-two expansion of $\Omega_U$ be calculated ?
\qed
\end{problem}

As a result of this paper,
we can answer this problem negatively.
In this paper, however,
we consider more general problems in the following forms.
Let $V$ and $W$ be optimal computers.

\begin{problem}\label{problem1}
Find a succinct equivalent characterization of
a total recursive function $f\colon \N^+\to \N$
which satisfies the condition:
For all $n\in\N^+$,
if $n$ and
the list of all halting inputs for $V$ of length at most $n$ are given,
then the first $n-f(n)-O(1)$ bits of the base-two expansion of $\Omega_W$
can be calculated.
\qed
\end{problem}

\begin{problem}\label{problem2}
Find a succinct equivalent characterization of
a total recursive function $f\colon \N^+\to \N$
which satisfies the condition:
For infinitely many $n\in\N^+$,
if $n$ and the list of all halting inputs for $V$ of length
at most $n$ are given,
then the first $n-f(n)-O(1)$ bits of
the base-two expansion of $\Omega_W$
can be calculated.
\qed
\end{problem}

Here $\N^+$ denotes the set of positive integers
and $\N=\left\{0\right\}\cup\N^+$.
Theorem~\ref{main1} and Theorem~\ref{main2} below are
two of the main results of this paper.
On the one hand,
Theorem~\ref{main1}
gives to Problem~\ref{problem1}
a solution that
the
total recursive
function
$f$ must satisfy
$\sum_{n=1}^\infty 2^{-f(n)}<\infty$,
which is the Kraft inequality in essence.
Note that the condition $\sum_{n=1}^\infty 2^{-f(n)}<\infty$ holds
for $f(n)=\lfloor(1+\varepsilon)\log_2 n \rfloor$
with an arbitrary computable real $\varepsilon>0$,
while this condition does not hold
for $f(n)=\lfloor\log_2 n\rfloor$.
On the other hand,
Theorem~\ref{main2}
gives to Problem~\ref{problem2}
a solution that
the
total recursive
function
$f$ must not be
bounded to the above.
Theorem~\ref{main2} also results in Corollary~\ref{cor} below,
which refutes Problem~\ref{motivation} completely.

It is also important to consider
whether the bound $n$ on the length of halting inputs
given in Fact~\ref{Chaitin} is tight or not.
We consider this problem in the following form:

\begin{problem}\label{problem3}
Find a succinct equivalent characterization of
a total recursive function $f\colon \N^+\to \N$
which satisfies the condition:
For all $n\in\N^+$,
if $n$ and the first $n$ bits of the base-two expansion of $\Omega_V$
are given,
then the list of all halting inputs for $W$ of length at most $n+f(n)-O(1)$
can be calculated.
\qed
\end{problem}

Theorem~\ref{main3},
which is one of the main results of this paper,
gives to Problem~\ref{problem3}
a solution that
the total recursive function
$f$ must be bounded to the above.
Thus, we see that
the bound $n$ on the length of halting inputs
given in Fact~\ref{Chaitin}
is tight up to an additive constant.

It is
well
known that
the base-two expansion of $\Omega_U$ and the halting problem of $U$
are Turing equivalent, i.e.,
$\Omega_U\equiv_T \Dom U$ holds,
where $\Dom U$ denotes the domain of definition of $U$.
This paper investigates an elaboration of the Turing equivalence.
For example,
consider the Turing reduction $\Omega_U\le_T \Dom U$,
which partly constitutes the Turing equivalence $\Omega_U\equiv_T \Dom U$.
The Turing reduction can be equivalent to the condition that
there exists an oracle deterministic Turing machine $M$
such that,
for all $n\in\N^+$,
\begin{equation}\label{Tr1}
  M^{\Dom U}(n)=\rest{\Omega_U}{n},
\end{equation}
where $\rest{\Omega_U}{n}$ denotes
the first $n$ bits of the base-two expansion of $\Omega_U$.
Let $g\colon \N^+\to \N$ and $h\colon \N^+\to \N$ be
total recursive functions.
Then
the condition \eqref{Tr1} can be elaborated to the condition that
there exists an oracle deterministic Turing machine $M$
such that,
for all $n\in\N^+$,
\begin{equation}\label{eTr1}
  M^{\reste{\Dom U}{g(n)}}(n)=\rest{\Omega_U}{h(n)},
\end{equation}
where $\rest{\Dom U}{g(n)}$ denotes
the set of all strings in $\Dom U$ of length at most $g(n)$.
This elaboration allows us
to consider the asymptotic behavior of $h$
which satisfies the condition \eqref{eTr1},
for a given $g$.
We might regard $g$
as the degree of
the relaxation of the restrictions on the computational resource
(i.e., on the oracle $\Dom U$)
and $h$ as the difficulty of the problem to solve.
Thus, even in the context of computability theory,
we can deal with the notion of asymptotic behavior
in a manner like in computational complexity theory
in some sense.
Theorem~\ref{main1}, a solution to Problem~\ref{problem1},
is obtained as a result of the investigation in this line,
and gives the upper bound of the function $h$ in the case of $g(n)=n$.

The other Turing reduction $\Dom U\le_T\Omega_U$,
which constitutes $\Omega_U\equiv_T \Dom U$,
is also elaborated in the same manner as above
to lead to
Theorem~\ref{main3},
a solution to Problem~\ref{problem3}.

Thus, in this paper, we study the relationship
between the base-two expansion of $\Omega$ and
the halting problem of an optimal computer
using a more rigorous and insightful notion than
the notion of Turing equivalence.
The paper is organized as follows.
We begin in Section~\ref{preliminaries} with
some preliminaries to
AIT.
We then prove Theorems~\ref{main1}, \ref{main2}, and \ref{main3}
in Sections~\ref{main result I}, \ref{main result II},
and \ref{main result III}, respectively.

\section{Preliminaries}
\label{preliminaries}

\subsection{Basic notation}
\label{basic notation}

We start with some notation about numbers and strings
which will be used in this paper.
$\#S$ is the cardinality of $S$ for any set $S$.
$\N=\left\{0,1,2,3,\dotsc\right\}$ is the set of natural numbers,
and $\N^+$ is the set of positive integers.
$\Z$ is the set of integers, and
$\Q$ is the set of rational numbers.
$\R$ is the set of real numbers.
Let $f\colon S\to\R$ with $S\subset\R$.
We say that $f$ is \textit{increasing} (resp., \textit{non-decreasing})
if $f(x)<f(y)$ (resp., $f(x)\le f(y)$) for all $x,y\in S$ with $x<y$.

Normally,
$O(1)$ denotes any function $f\colon \N^+\to\R$ such that
there is $C\in\R$ with the property that
$\abs{f(n)}\le C$ for all $n\in\N^+$.

$\X=
\left\{
  \lambda,0,1,00,01,10,11,000,\dotsc
\right\}$
is the set of finite binary strings
where $\lambda$ denotes the \textit{empty string},
and $\X$ is ordered as indicated.
We identify any string in $\X$ with a natural number in this order,
i.e.,
we consider $\varphi\colon \X\to\N$ such that $\varphi(s)=1s-1$
where the concatenation $1s$ of strings $1$ and $s$ is regarded
as a dyadic integer,
and then we identify $s$ with $\varphi(s)$.
For any $s \in \X$, $\abs{s}$ is the \textit{length} of $s$.
A subset $S$ of $\X$ is called
\textit{prefix-free}
if no string in $S$ is a prefix of another string in $S$.
For any subset $S$ of $\X$ and any $n\in\Z$,
we denote by $\rest{S}{n}$
the set $\{s\in S\mid \abs{s}\le n\}$.
Note that $\rest{S}{n}=\emptyset$
for every subset $S$ of $\X$ and every negative integer $n\in\Z$.
$\XI$ is the set of infinite binary strings,
where an infinite binary string is
infinite to the right but finite to the left.
For any partial function $f$,
the domain of definition of $f$ is denoted by $\Dom f$.
We write ``r.e.'' instead of ``recursively enumerable.''

Let $\alpha$ be an arbitrary real number.
$\lfloor \alpha \rfloor$ is the greatest integer less than or equal to $\alpha$,
and $\lceil \alpha \rceil$ is the smallest integer greater than or equal to $\alpha$.
For any $n\in\N^+$,
we denote by $\rest{\alpha}{n}\in\X$
the first $n$ bits of the base-two expansion of
$\alpha - \lfloor \alpha \rfloor$ with infinitely many zeros.
For example,
in the case of $\alpha=5/8$,
$\rest{\alpha}{6}=101000$.
On the other hand,
for any non-positive integer $n\in\Z$,
we set $\rest{\alpha}{n}=\lambda$.

A real number $\alpha$ is called \textit{r.e.}~if
there exists a total recursive function $f\colon\N^+\to\Q$ such that
$f(n)\le\alpha$ for all $n\in\N^+$ and $\lim_{n\to\infty} f(n)=\alpha$.
An r.e.~real number is also called
a \textit{left-computable} real number.


\subsection{Algorithmic information theory}
\label{ait}

In the following
we concisely review some definitions and results of
algorithmic information theory
\cite{C75,C87b}.
A \textit{computer} is a partial recursive function
$C\colon \X\to \X$
such that
$\Dom C$ is a prefix-free set.
For each computer $C$ and each $s \in \X$,
$H_C(s)$ is defined by
$H_C(s) =
\min
\left\{\,
  \abs{p}\,\big|\;p \in \X\>\&\>C(p)=s
\,\right\}$
(may be $\infty$).
A computer $U$ is said to be \textit{optimal} if
for each computer $C$ there exists $d\in\N$
with the following property;
if $p\in\Dom C$, then there is $q$ for which
$U(q)=C(p)$ and $\abs{q}\le\abs{p}+d$.
It is easy to see that there exists an optimal computer.
We choose a particular optimal computer $U$
as the standard one for use,
and define $H(s)$ as $H_U(s)$,
which is referred to as
the \textit{program-size complexity} of $s$,
the \textit{information content} of $s$, or
the \textit{Kolmogorov complexity} of $s$
\cite{G74,L74,C75}.
It follows that
for every computer $C$ there exists $d\in\N$ such that,
for every $s\in\X$,
\begin{equation}\label{minimal}
  H(s)\le H_C(s)+d.
\end{equation}
Based on this we can show
that there exists $c\in\N$ such that,
for every $s\in\X$,
\begin{equation}\label{eq: fas}
  H(s)\le 2\abs{s}+c.
\end{equation}
Using \eqref{minimal}
we can also show that,
for every partial recursive function
$\Psi\colon \X\to \X$,
there exists $c\in\N$ such that,
for every $s \in \Dom \Psi$,
\begin{equation}\label{Psi}
  H(\Psi(s))\le H(s)+c.
\end{equation}
For any $s\in\X$,
we define $s^*$ as $\min\{\,p\in\X\mid U(p)=s\}$,
i.e., the first element in the ordered set $\X$
of all strings $p$ such that $U(p)=s$.
Then, $\abs{s^*}=H(s)$ for every $s\in\X$.
For any $s,t\in\X$,
we define $H(s,t)$ as $H(b(s,t))$,
where $b\colon \X\times \X\to \X$ is
a particular bijective total recursive function.
Note also that, for every $n \in \N$,
$H(n)$ is $H(\text{the $n$th element of }\X)$.


\begin{definition}[Chaitin $\Omega$ number, Chaitin \cite{C75}]
For any optimal computer $V$,
the halting probability $\Omega_V$ of $V$ is defined
by
\begin{equation*}
  \Omega_V=\sum_{p\in\Dom V}2^{-\abs{p}}.
\end{equation*}
\qed
\end{definition}

For every optimal computer $V$,
since $\Dom V$ is prefix-free,
$\Omega_V$ converges and $0<\Omega_V\le 1$.
%
%
For any $\alpha\in\R$,
we say that $\alpha$ is \textit{weakly Chaitin random}
if there exists $c\in\N$ such that
$n-c\le H(\rest{\alpha}{n})$ for all $n\in\N^+$
\cite{C75,C87b}.
%

\begin{theorem}[Chaitin \cite{C75}]
For every optimal computer $V$,
$\Omega_V$ is weakly Chaitin random.
\qed
\end{theorem}

Therefore $0<\Omega_V<1$ for every optimal computer $V$.
%
%
For any $\alpha\in\R$,
we say that $\alpha$ is
\textit{Chaitin random} if
$\lim_{n\to \infty} H(\rest{\alpha}{n})-n=\infty$
\cite{C75,C87b}.
We can
then
show the following theorem
(see
Chaitin
\cite{C87b} for the proof and historical detail).

\begin{theorem}\label{Cr}
For every $\alpha\in\R$,
$\alpha$ is weakly Chaitin random if and only if
$\alpha$ is Chaitin random.
\qed
\end{theorem}

The following is an important result on random r.e.~reals.

\begin{theorem}[Calude, et al.~\cite{CHKW01},
Ku\v{c}era and Slaman \cite{KS01}]\label{re-Omega}
For every $\alpha\in(0,1)$,
$\alpha$ is r.e.~and weakly Chaitin random if and only if
there exists an optimal computer $V$ such that $\alpha=\Omega_V$.
\qed
\end{theorem}

\section{\boldmath
Elaboration I of the Turing reduction $\Omega_U\le_T\Dom U$}
\label{main result I}

\begin{theorem}[main result I]\label{main1}
Let $V$ and $W$ be optimal computers,
and let $f\colon \N^+\to \N$ be a total recursive function.
Then the following two conditions are equivalent:
\begin{enumerate}
\item There exist an oracle deterministic Turing machine $M$
  and $c\in\N$ such that,
  for all $n\in\N^+$,
  $M^{\reste{\Dom V}{n}}(n)=\rest{\Omega_W}{n-f(n)-c}$.
\item $\sum_{n=1}^\infty 2^{-f(n)}<\infty$.\qed
\end{enumerate}
\end{theorem}

Theorem~\ref{main1} follows from
Theorem~\ref{I-re} and Theorem~\ref{I-random} below,
and Theorem~\ref{re-Omega}.

\begin{theorem}\label{I-re}
Let $\alpha$ be an r.e.~real,
and let $V$ be an optimal computer.
For every total recursive function $f\colon \N^+\to \N$,
if $\sum_{n=1}^\infty 2^{-f(n)}<\infty$,
then there exist an oracle deterministic Turing machine $M$ and $c\in\N$
such that,
for all $n\in\N^+$,
$M^{\reste{\Dom V}{n}}(n)=\rest{\alpha}{n-f(n)-c}$.
\qed
\end{theorem}

\begin{theorem}\label{I-random}
Let $\alpha$ be a real which is weakly Chaitin random,
and let $V$ be an optimal computer.
For every total recursive function $f\colon \N^+\to \N$, 
if there exists an oracle deterministic Turing machine $M$
such that, for all $n\in\N^+$,
$M^{\reste{\Dom V}{n}}(n)=\rest{\alpha}{n-f(n)}$,
then $\sum_{n=1}^\infty 2^{-f(n)}<\infty$.
\qed
\end{theorem}

The proofs of
Theorem~\ref{I-re} and Theorem~\ref{I-random}
are given in the next two subsections, respectively.

Note that, as a variant of Theorem~\ref{main1},
we can prove the following theorem as well,
in a similar manner to the proof of Theorem~\ref{main1}.

\begin{theorem}[variant of the main result I]\label{variant-main1}
Let $V$ and $W$ be optimal computers,
and let $f\colon \N^+\to \N$ be a total recursive function.
Then the following two conditions are equivalent:
\begin{enumerate}
\item There exist an oracle deterministic Turing machine $M$
  and $c\in\N$ such that,
  for all $n\in\N^+$,
  $M^{\reste{\Dom V}{n+f(n)+c}}(n)=\rest{\Omega_W}{n}$.
\item $\sum_{n=1}^\infty 2^{-f(n)}<\infty$.\qed
\end{enumerate}
\end{theorem}

\subsection{The proof of Theorem~\ref{I-re}}
\label{proofIre}

In order to prove Theorem~\ref{I-re},
we need Theorem~\ref{KC} and Corollary~\ref{oc-c} below.


\begin{theorem}[Kraft-Chaitin Theorem, Chaitin \cite{C75}]\label{KC}
Let $f\colon \N^+\to \N$ be a total recursive function
such that $\sum_{n=1}^\infty 2^{-f(n)}\le 1$.
Then there exists a total recursive function $g\colon \N^+\to \X$
such that
(i)
the function
$g$ is an injection,
(ii) the set $\{\,g(n)\mid n\in\N^+\}$ is prefix-free, and
(iii) $\abs{g(n)}=f(n)$ for all $n\in\N^+$.
\qed
\end{theorem}

Let $M$ be
a
deterministic Turing machine
with the input and output alphabet $\{0,1\}$, 
and let
$C$
be a computer.
We say that $M$ \textit{computes} $C$ if the following holds:
for every $p\in\X$,
when $M$ starts with the input $p$,
(i) $M$ halts and outputs $C(p)$ if $p\in\Dom C$;
(ii) $M$ does not halt forever otherwise.
We use this convention on the computation of a computer
by a deterministic Turing machine
throughout the rest of this paper.
Thus, we exclude the possibility that
there is $p\in\X$ such that,
when $M$ starts with the input $p$,
$M$ halts
but
$p\notin\Dom C$.

\begin{theorem}\label{weaksim}
Let $V$ be an optimal computer.
Then, for every computer $C$ there exists $d\in\N$ such that,
for every $p\in\X$,
if $p$ and the list of all halting inputs for $V$ of length
at most $\abs{p}+d$ are given,
then the halting problem of the input $p$ for $C$ can be solved.
\end{theorem}

\begin{proof}
Let $M$ be a deterministic Turing machine which computes a computer $C$.
For each $p\in\X$,
let $h_M(p)$ be the computation history of $M$
from the initial configuration with input $p$,
and let $\bin_M(p)\in\X\cup\XI$ be the binary representation of $h_M(p)$
in a certain format.
Note that $\bin_M(p)\in\X$ if and only if $p\in\Dom C$
for every $p\in\X$,
by our convention on the computation of a computer
by a deterministic Turing machine.
We consider the computer $D$ such that
(i) $\Dom D=\Dom C$ and
(ii) $D(p)=\bin_M(p)$ for every $p\in\Dom C$.
It is easy to see that such a computer $D$ exists.
Then,
since $V$ is an optimal computer,
from the definition of
optimality
there exists $d\in\N$ with the following property;
if $p\in\Dom D$, then there is $q$ for which
$V(q)=D(p)$ and $\abs{q}\le\abs{p}+d$.

Given $p\in\X$ and
the list $\{q_1,\dots,q_L\}$ of all halting inputs for $V$ of length
at most $\abs{p}+d$,
one first calculates the finite set
$S_p=\{\,V(q_i)\mid i=1,\dots,L\,\}$.
One then checks whether $\bin_M(p)\in S_p$ or not.
This can be possible since $S_p$ is a finite subset of $\X$.
In the case of $\bin_M(p)\in S_p$,
$\bin_M(p)\in\X$ and therefore $p\in\Dom C$.
On the other hand,
if $p\in\Dom C$,
then there is $q$ such that
$V(q)=\bin_M(p)$ and $\abs{q}\le\abs{p}+d$,
and therefore $q\in\{q_1,\dots,q_L\}$ and $\bin_M(p)\in S_p$.
Thus, $p\notin\Dom C$ in the case of $\bin_M(p)\notin S_p$.
\end{proof}

\begin{remark}
A partial recursive function $u\colon \X\to \X$ is called \textit{optimal}
if for every partial recursive function $f\colon \X\to \X$
there exists $d\in\N$ such that, for every $p\in\Dom f$,
there is $q\in\Dom u$ for which $u(q)=f(p)$ and $\abs{q}\le\abs{p}+d$.
An optimal partial recursive function $u\colon \X\to \X$
is used to define the notion of plain program-size complexity.
Obviously,
we can show that the same theorem as Theorem~\ref{weaksim} holds
between
the halting problem of
any optimal partial recursive function $u\colon \X\to \X$
and one of any partial recursive function $f\colon \X\to \X$.
\qed
\end{remark}

As a corollary of Theorem~\ref{weaksim} above
we obtain the following.

\begin{corollary}\label{oc-c}
Let $V$ be an optimal computer.
Then, for every computer $C$
there exist an oracle deterministic Turing machine $M$ and $d\in\N$
such that, for all $n\in\N^+$,
$M^{\reste{\Dom V}{n+d}}(n)=\rest{\Dom C}{n}$,
where the finite subset $\rest{\Dom C}{n}$ of $\X$ is
represented as a finite binary string in a certain format.
\qed
\end{corollary}

Based on Theorem~\ref{KC} and Corollary~\ref{oc-c},
Theorem~\ref{I-re} is proved as follows.

\begin{proof}[Proof of Theorem~\ref{I-re}]
Let $\alpha$ be an r.e.~real,
and let $V$ be an optimal computer.
For an arbitrary total recursive function $f\colon \N^+\to \N$,
assume that $\sum_{n=1}^\infty 2^{-f(n)}<\infty$.
In the case of $\alpha\in\Q$,
the result is obvious.
Thus,
in what follows,
we assume that $\alpha\notin\Q$ and therefore
the base-two expansion of
$\alpha - \lfloor \alpha \rfloor$ is unique
and contains infinitely many ones.

Since $\sum_{n=1}^\infty 2^{-f(n)}<\infty$,
there exists $d_0\in\N$ such that
$\sum_{n=1}^\infty 2^{-f(n)-d_0}\le 1$.
Hence, by the Kraft-Chaitin Theorem, i.e., Theorem~\ref{KC},
there exists a total recursive function
$g\colon \N^+\to \X$ such that
(i) the function $g$ is an injection,
(ii) the set $\{\,g(n)\mid n\in\N^+\}$ is prefix-free, and
(iii) $\abs{g(n)}=f(n)+d_0$ for all $n\in\N^+$.
On the other hand, since $\alpha$ is r.e.,
there exists a total recursive function $h\colon\N^+\to\Q$ such that
$h(k)\le\alpha$ for all $k\in\N^+$ and $\lim_{k\to\infty} h(k)=\alpha$.

Now,
let us consider the following computer $C$.
For each $n\in\N^+$, $p,s\in\X$ and $l\in\N$ such that
$U(p)=l$,
$g(n)ps\in\Dom C$ if and only if
(i) $\abs{g(n)ps}=n-l$ and
(ii) $0.s<h(k)-\lfloor \alpha \rfloor$ for some $k\in\N^+$.
It is easy to see that such a computer $C$ exists.
Then, by Corollary~\ref{oc-c},
there exist an oracle deterministic Turing machine $M$ and $d\in\N$
such that, for all $n\in\N^+$,
$M^{\reste{\Dom V}{n+d}}(n)=\rest{\Dom C}{n}$,
where the finite subset $\rest{\Dom C}{n}$ of $\X$ is
represented as a finite binary string in a certain format.
We then see that,
for every $n\in\N^+$ and $s\in\X$ such that
$\abs{s}=n-\abs{g(n)}-d-\abs{d^*}$,
\begin{equation}\label{decision-condition}
  g(n)d^*s\in\Dom C
  \text{ if and only if }
  s\le \rest{\alpha}{n-\abs{g(n)}-d-\abs{d^*}},
\end{equation}
where $s$ and $\rest{\alpha}{n-\abs{g(n)}-d-\abs{d^*}}$ are
regarded as a dyadic integer.
Then, by the following procedure,
we see that
there exist an oracle deterministic Turing machine $M_1$ and $c\in\N$
such that,
for all $n\in\N^+$,
$M_1^{\reste{\Dom V}{n}}(n)=\rest{\alpha}{n-f(n)-c}$.
Note here that $\abs{g(n)}=f(n)+d_0$ for all $n\in\N^+$
and also $H(d)=\abs{d^*}$.

Given $n$ and $\reste{\Dom V}{n}$ with $n>d$,
one first checks whether $n-\abs{g(n)}-d-H(d)\le 0$ holds.
If this holds then one outputs $\lambda$.
If this does not hold,
one then calculates the finite set $\rest{\Dom C}{n-d}$
by simulating the computation of $M$
with the input $n-d$ and the oracle $\rest{\Dom V}{n}$.
Then, based on \eqref{decision-condition},
one determines $\rest{\alpha}{n-\abs{g(n)}-d-H(d)}$
by checking whether $g(n)d^*s\in\Dom C$ holds or not for each
$s\in\X$ with $\abs{s}=n-\abs{g(n)}-d-H(d)$.
This is possible since $\abs{g(n)d^*s}=n-d$
for every $s\in\X$ with $\abs{s}=n-\abs{g(n)}-d-H(d)$.
Finally, one outputs $\rest{\alpha}{n-\abs{g(n)}-d-H(d)}$.
\end{proof}

\subsection{The proof of Theorem~\ref{I-random}}
\label{proofIrandom}

In order to prove Theorem~\ref{I-random},
we need Theorem~\ref{time} and
the Ample Excess Lemma (i.e., Theorem~\ref{AEL}) below.

Let $M$ be an arbitrary deterministic Turing machine
with the input alphabet $\{0,1\}$.
We define
$L_M=
\min\{\,\abs{p}\mid p\in\X\text{ \& $M$ halts on input $p$}\,\}$
(may be $\infty$).
For any $n\ge L_M$,
we define
$T^M_n$
as the maximum running time of $M$
on all halting inputs of length at most $n$.

\begin{theorem}\label{time}
Let $V$ be an optimal computer,
and let $M$ be a deterministic Turing machine which computes $V$.
Then
$n=H(T^M_n,n)+O(1)=H(T^M_n)+O(1)$
for all $n\ge L_M$.
\qed
\end{theorem}

Note that
Solovay \cite{Sol75} showed a similar result to Theorem~\ref{time}
for $h_n=\#\{p\in\Dom V\mid \abs{p}\le n\}$
in place of $T^M_n$.
On the other hand,
Chaitin showed a similar result to Theorem~\ref{time}
for $p\in\Dom V$ such that
$\abs{p}\le n$ and
the running time of $M$ on the input $p$ equals to $T^M_n$,
in place of $T^M_n$ (see Note in Section~8.1 of Chaitin \cite{C87b}).
We include the proof of Theorem~\ref{time}
in Appendix~\ref{proof-time}
for completeness.

Miller and Yu \cite{MY08} recently strengthened Theorem~\ref{Cr}
to the following form. 

\begin{theorem}[Ample Excess Lemma, Miller and Yu \cite{MY08}]\label{AEL}
For every $\alpha\in\R$,
$\alpha$ is weakly Chaitin random if and only if
$\sum_{n=1}^{\infty} 2^{n-H(\reste{\alpha}{n})}<\infty$.
\qed
\end{theorem}

Then the proof of Theorem~\ref{I-random} is as follows.

\begin{proof}[Proof of Theorem~\ref{I-random}]
Let $\alpha$ be a real which is weakly Chaitin random.
Let $V$ be an optimal computer,
and let $M$ be a deterministic Turing machine which computes $V$.
For an arbitrary total recursive function $f\colon \N^+\to \N$,
assume that there exists an oracle deterministic Turing machine $M_0$
such that, for all $n\in\N^+$,
$M_0^{\reste{\Dom V}{n}}(n)=\rest{\alpha}{n-f(n)}$.
Note that,
given $(T_{n}^M,n)$ with $n\ge L_M$,
one can calculate the finite set $\rest{\Dom V}{n}$
by simulating the computation of $M$ with the input $p$
until at most $T^M_n$ steps,
for each $p\in\X$ with $\abs{p}\le n$.
Thus, we see that
there exists a partial recursive function
$\Psi\colon \N\times\N^+\to \X$ such that,
for all $n\ge L_M$,
$\Psi(T^M_n,n)=\rest{\alpha}{n-f(n)}$.
It follows from \eqref{Psi} that
$H(\rest{\alpha}{n-f(n)})
\le
H(T^M_n,n)+O(1)$
for all $n\ge L_M$.
Thus, by Theorem~\ref{time} we have
\begin{equation}\label{complexity-bound}
  H(\rest{\alpha}{n-f(n)})\le n +O(1)
\end{equation}
for all $n\in\N^+$.

In the case where the function $n-f(n)$ of $n$ is bounded to the above,
there exists $c\in\N$ such that, for every $n\in\N^+$, $-f(n)\le c-n$,
and therefore $\sum_{n=1}^{\infty}2^{-f(n)}\le 2^{c}$.
Thus, in what follows, we assume that
the function $n-f(n)$ of $n$ is not bounded to the above.

We define
a function
$g\colon \N^+\to \Z$ 
by $g(n)=\max\{k-f(k)\mid 1\le k\le n\}$.
It follows that
the function $g$ is non-decreasing and
$\lim_{n\to\infty} g(n)=\infty$.
Thus we can choose an enumeration
$n_1,n_2,n_3, \dotsc$ of
the countably infinite set
$\{n\in\N^+\mid n\ge 2\ \&\ 0\le g(n-1)<g(n)\}$
with $n_j<n_{j+1}$.
It is then easy to see that
$g(n_{j})=n_{j}-f(n_{j})$
and $1\le n_{j}-f(n_{j})<n_{j+1}-f(n_{j+1})$ hold for all $j$.
On the other hand,
since $\alpha$ is weakly Chaitin random,
using the Ample Excess Lemma, i.e., Theorem~\ref{AEL},
we have $\sum_{n=1}^{\infty} 2^{n-H(\reste{\alpha}{n})}<\infty$.
Thus, using \eqref{complexity-bound}
we see that
\begin{equation}\label{subseqinfty}
  \sum_{j=1}^{\infty} 2^{-f(n_{j})}
  \le
  \sum_{j=1}^{\infty} 2^{n_{j}-f(n_{j})-H(\reste{\alpha}{n_{j}-f(n_{j})})+O(1)}
  \le
  \sum_{n=1}^{\infty} 2^{n-H(\reste{\alpha}{n})+O(1)}
  <\infty. 
\end{equation}

On the other hand,
it is easy to see that
(i) $g(n)\ge n-f(n)$ for every $n\in\N^+$, and
(ii) $g(n)=g(n_{j})$
for every $j$ and $n$ with $n_{j}\le n<n_{j+1}$.
Thus, for each $k\ge 2$, it is shown that
\begin{equation*}
\begin{split}
  \sum_{n=n_{1}}^{n_{k}-1}2^{-f(n)}
  &\le
  \sum_{n=n_{1}}^{n_{k}-1}2^{g(n)-n}
  =
  \sum_{j=1}^{k-1}\sum_{n=n_{j}}^{n_{j+1}-1}2^{g(n)-n}
  =
  \sum_{j=1}^{k-1}2^{g(n_{j})}
  \sum_{n=n_{j}}^{n_{j+1}-1}2^{-n}\\
  &=
  \sum_{j=1}^{k-1}2^{n_{j}-f(n_{j})}
  2^{-n_{j}+1}\left(1-2^{-n_{j+1}+n_{j}}\right)
  <
  2\sum_{j=1}^{k-1}2^{-f(n_{j})}.
\end{split}
\end{equation*}
Thus, 
using \eqref{subseqinfty}
we see that
$\lim_{k\to\infty}\sum_{n=n_{1}}^{n_{k}-1}2^{-f(n)}<\infty$.
Since $2^{-f(n)}>0$ for all $n\in\N^+$
and $\lim_{j\to\infty}n_{j}=\infty$,
we have $\sum_{n=1}^{\infty}2^{-f(n)}<\infty$.
\end{proof}

\section{\boldmath
Elaboration II of the Turing reduction $\Omega_U\le_T\Dom U$}
\label{main result II}

\begin{theorem}[main result II]\label{main2}
Let $V$ and $W$ be optimal computers,
and let $f\colon \N^+\to \N$ be a total recursive function.
Then the following two conditions are equivalent:
\begin{enumerate}
\item There exist an oracle deterministic Turing machine $M$
  and $c\in\N$ such that,
  for infinitely many $n\in\N^+$,
  $M^{\reste{\Dom V}{n}}(n)=\rest{\Omega_W}{n-f(n)-c}$.
\item The function $f$ is not bounded to the above.\qed
\end{enumerate}
\end{theorem}

The proof of Theorem~\ref{main2} is given
in Subsection~\ref{proofII} below.
By setting $f(n)=0$ and $W=V$ in Theorem~\ref{main2},
we obtain the following.

\begin{corollary}\label{cor}
Let $V$ be an optimal computer.
Then, for every $c\in\N$,
there does not exist an oracle deterministic Turing machine $M$ such that,
for infinitely many $n\in\N^+$,
$M^{\reste{\Dom V}{n+c}}(n)=\rest{\Omega_V}{n}$.
\qed
\end{corollary}

Note that, as a variant of Theorem~\ref{main2},
we can prove the following theorem as well,
in a similar manner to the proof of Theorem~\ref{main2}.

\begin{theorem}[variant of the main result II]\label{variant-main2}
Let $V$ and $W$ be optimal computers,
and let $f\colon \N^+\to \N$ be a total recursive function.
Then the following two conditions are equivalent:
\begin{enumerate}
\item There exist an oracle deterministic Turing machine $M$
  and $c\in\N$ such that,
  for infinitely many $n\in\N^+$,
  $M^{\reste{\Dom V}{n+f(n)+c}}(n)=\rest{\Omega_W}{n}$.
\item The function $f$ is not bounded to the above.\qed
\end{enumerate}
\end{theorem}

\subsection{The proof of Theorem~\ref{main2}}
\label{proofII}

Theorem~\ref{main2} follows from
Theorem~\ref{II-re} and Theorem~\ref{II-random} below,
and Theorem~\ref{re-Omega}.

\begin{theorem}\label{II-re}
Let $\alpha$ be an r.e.~real,
and let $V$ be an optimal computer.
For every total recursive function $f\colon \N^+\to \N$,
if the function $f$ is not bounded to the above,
then there exist an oracle deterministic Turing machine $M$ and $c\in\N$
such that,
for infinitely many $n\in\N^+$,
$M^{\reste{\Dom V}{n}}(n)=\rest{\alpha}{n-f(n)-c}$.
\qed
\end{theorem}

In order to prove Theorem~\ref{II-re},
we need Lemma~\ref{H-unboundedf} below.
It is easy to show Lemma~\ref{H-unboundedf}.
For completeness, however,
we include the proof of Lemma~\ref{H-unboundedf}
in Appendix~\ref{proof-H-unboundedf}.

\begin{lemma}\label{H-unboundedf}
Let $f\colon \N^+\to \N$ be a total recursive function.
If the function $f$ is not bounded to the above,
then $H(n)\le f(n)$ for infinitely many $n\in\N^+$.
\qed
\end{lemma}

Based on Lemma~\ref{H-unboundedf} and Corollary~\ref{oc-c},
Theorem~\ref{II-re} is proved as follows.

\begin{proof}[Proof of Theorem~\ref{II-re}]
Let $\alpha$ be an r.e.~real,
and let $V$ be an optimal computer.
For an arbitrary total recursive function $f\colon \N^+\to \N$,
assume that the function $f$ is not bounded to the above.
In the case of $\alpha\in\Q$,
the result is obvious.
Thus,
in what follows,
we assume that $\alpha\notin\Q$ and therefore
the base-two expansion of
$\alpha - \lfloor \alpha \rfloor$ is unique
and contains infinitely many ones.

Since the total recursive function $f$ is not bounded to the above,
by Lemma~\ref{H-unboundedf}
we see that $H(n)\le f(n)$ for infinitely many $n\in\N^+$.
Note also that $\lim_{n\to\infty} n-H(n)=\infty$.
This is because $H(n)\le 2\log_2 n+O(1)$ holds for all $n\in\N^+$
by \eqref{eq: fas}.
On the other hand, since $\alpha$ is r.e.,
there exists a total recursive function $g\colon\N^+\to\Q$ such that
$g(k)\le\alpha$ for all $k\in\N^+$ and $\lim_{k\to\infty} g(k)=\alpha$.

Let us consider the following computer $C$.
For each $p,q,s\in\X$ and $n,l\in\N$ such that
$U(p)=n$ and $U(q)=l$,
$pqs\in\Dom C$ if and only if
(i) $\abs{pqs}=n-l$ and
(ii) $0.s<g(k) - \lfloor \alpha \rfloor$ for some $k\in\N^+$.
It is easy to see that such a computer $C$ exists.
Then, by Corollary~\ref{oc-c},
there exist an oracle deterministic Turing machine $M$ and $d\in\N$
such that, for all $n\in\N^+$,
$M^{\reste{\Dom V}{n+d}}(n)=\rest{\Dom C}{n}$,
where the finite subset $\rest{\Dom C}{n}$ of $\X$ is
represented as a finite binary string in a certain format.
We then see that,
for every $n\in\N^+$ and $p,s\in\X$ such that
$U(p)=n$ and $\abs{s}=n-\abs{p}-d-\abs{d^*}$,
\begin{equation}\label{decision-conditionII}
  pd^*s\in\Dom C
  \text{ if and only if }
  s\le \rest{\alpha}{n-\abs{p}-d-\abs{d^*}},
\end{equation}
where $s$ and $\rest{\alpha}{n-\abs{p}-d-\abs{d^*}}$ are
regarded as a dyadic integer.
Then,
by the following procedure,
we see that
there exist an oracle deterministic Turing machine $M_1$ and $c\in\N$
such that,
for infinitely many $n\in\N^+$,
$M_1^{\reste{\Dom V}{n}}(n)=\rest{\alpha}{n-f(n)-c}$.
Note here that $H(d)=\abs{d^*}$.

Given $n$ and $\reste{\Dom V}{n}$ with $n>d$,
one first tries to find $p\in\X$ which satisfies that
(i) $U(p)=n$, (ii) $\abs{p}\le f(n)$, and (iii) $n-\abs{p}-d-H(d)\ge 1$.
One can find such a string $p$ for the cases of infinitely many $n\in\N^+$.
This is because $H(k)\le f(k)$ holds for infinitely many $k\in\N^+$
and $\lim_{k\to\infty} k-H(k)=\infty$.
If such a string $p$ is found,
one then calculates the finite set $\rest{\Dom C}{n-d}$
by simulating the computation of $M$
with the input $n-d$ and the oracle $\rest{\Dom V}{n}$.
Then, based on \eqref{decision-conditionII},
one determines $\rest{\alpha}{n-\abs{p}-d-H(d)}$
by checking whether $pd^*s\in\Dom C$ holds or not for each
$s\in\X$ with $\abs{s}=n-\abs{p}-d-H(d)$.
This is possible since $\abs{pd^*s}=n-d$
for every $s\in\X$ with $\abs{s}=n-\abs{p}-d-H(d)$.
Finally, one calculates and outputs $\rest{\alpha}{n-f(n)-d-H(d)}$.
This is possible since $n-f(n)-d-H(d)\le n-\abs{p}-d-H(d)$.
\end{proof}

\begin{theorem}\label{II-random}
Let $\alpha$ be a real which is weakly Chaitin random,
and let $V$ be an optimal computer.
For every total recursive function $f\colon \N^+\to \N$, 
if there exists an oracle deterministic Turing machine $M$ such that,
for infinitely many $n\in\N^+$,
$M^{\reste{\Dom V}{n}}(n)=\rest{\alpha}{n-f(n)}$,
then the function $f$ is not bounded to the above.
\qed
\end{theorem}

Using \eqref{Psi}, Theorem~\ref{time} and Theorem~\ref{Cr},
we can prove Theorem~\ref{II-random} as follows.

\begin{proof}[Proof of Theorem~\ref{II-random}]
Let $\alpha$ be a real which is weakly Chaitin random.
Let $V$ be an optimal computer,
and let $M$ be a deterministic Turing machine which computes $V$.
For an arbitrary total recursive function $f\colon \N^+\to \N$,
assume that there exists an oracle deterministic Turing machine $M_0$
such that, for infinitely many $n\in\N^+$,
$M_0^{\reste{\Dom V}{n}}(n)=\rest{\alpha}{n-f(n)}$.
Note that,
given $(T_{n}^M,n)$ with $n\ge L_M$,
one can calculate the finite set $\rest{\Dom V}{n}$
by simulating the computation of $M$ with the input $p$
until at most $T^M_n$ steps,
for each $p\in\X$ with $\abs{p}\le n$.
Thus, we see that
there exists a partial recursive function
$\Psi\colon \N\times\N^+\to \X$ such that,
for infinitely many $n\ge L_M$,
$\Psi(T^M_n,n)=\rest{\alpha}{n-f(n)}$.
It follows from \eqref{Psi} that
$H(\rest{\alpha}{n-f(n)})
\le
H(T^M_n,n)+O(1)$
for infinitely many $n\ge L_M$.
Thus, by Theorem~\ref{time} we see that
there exists an infinite subset $S$ of $\N^+$
such that
\begin{equation}\label{complexity-boundII}
  H(\rest{\alpha}{n-f(n)})\le n +O(1)
\end{equation}
for all $n\in S$.

In the case where
the function $n-f(n)$ of $n$ is bounded to the above on $S$,
there exists $c\in\N$ such that, for every $n\in S$, $n-c\le f(n)$,
and therefore the function $f$ itself is not bounded to the above.
Thus, in what follows, we assume that
the function $n-f(n)$ of $n$ is not bounded to the above on $S$.
Thus we can choose a sequence $n_1,n_2,n_3, \dotsc$ in $S$
such that $1\le n_{j}-f(n_{j})<n_{j+1}-f(n_{j+1})$
for all $j\in\N^+$.
It follows from \eqref{complexity-boundII} that
$H(\rest{\alpha}{n_{j}-f(n_{j})})-(n_{j}-f(n_{j}))
\le f(n_{j}) +O(1)$ for all $j\in\N^+$.
On the other hand,
since $\alpha$ is weakly Chaitin random,
it follows from Theorem~\ref{Cr} that
$\lim_{n\to \infty} H(\rest{\alpha}{n})-n=\infty$.
This implies that $\lim_{j\to \infty} f(n_{j})=\infty$.
Hence, the function $f$ is not bounded to the above.
\end{proof}

\section{\boldmath
Elaboration of the Turing reduction $\Dom U\le_T\Omega_U$}
\label{main result III}

\begin{theorem}[main result III]\label{main3}
Let $V$ and $W$ be optimal computers,
and let $f\colon \N^+\to \N$ be a total recursive function.
Then the following two conditions are equivalent:
\begin{enumerate}
\item There exist an oracle deterministic Turing machine $M$
  and $c\in\N$ such that,
  for all $n\in\N^+$,
  $M^{\{\reste{\Omega_V}{n}\}}(n)=\rest{\Dom W}{n+f(n)-c}$,
  where the finite subset $\rest{\Dom W}{n+f(n)-c}$ of $\X$ is
  represented as a finite binary string in a certain format.
\item The function $f$ is bounded to the above.\qed
\end{enumerate}
\end{theorem}

In order to prove the implication (i) $\Rightarrow$ (ii) of
Theorem~\ref{main3},
we need Theorem~\ref{Miller} below.
For the purpose of understanding the statement of Theorem~\ref{Miller},
we concisely review some definitions and results of
the theory of relative randomness.
See e.g.~\cite{N09,DH09} for the detail of the theory.

An \textit{oracle computer} is an oracle deterministic Turing machine $M$
with the input and output alphabet $\{0,1\}$
such that, for every subset $A$ of $\X$,
the domain of definition of
$M^A$ is
a prefix-free set.
For each oracle computer $M$, each subset $A$ of $\X$,
and each $s \in \X$,
$H_M^A(s)$ is defined by
$H_M^A(s) =
\min
\left\{\,
  \abs{p}\,\big|\;p \in \X\>\&\>M^A(p)=s
\,\right\}$
(may be $\infty$).
An oracle computer $R$ is said to be \textit{optimal} if
for every oracle computer $M$ there exists $d\in\N$ such that,
for every subset $A$ of $\X$ and every $s\in\X$,
\begin{equation}\label{oracle-minimal}
  H_R^A(s)\le H_M^A(s)+d.
\end{equation}
It is then easy to see that there exists an optimal oracle computer.
%
%
For any $\alpha\in\R$,
we say that $\alpha$ is \textit{$2$-random}
if there exist an optimal oracle computer $R$ and $c\in\N$
such that
$n-c\le H_R^{\Dom U}(\rest{\alpha}{n})$ for all $n\in\N^+$.
Recall here that
$U$ is the optimal computer used to define $H(s)$.

For any $\alpha\in\R$,
we say that $\alpha$ is \textit{strongly Chaitin random}
if there exists $c\in\N$ such that,
for infinitely many $n\in\N^+$,
$n+H(n)-c\le H(\rest{\alpha}{n})$.
J. Miller recently showed the following theorem.
See \cite{DH09,N09} for the detail.

\begin{theorem}[J. Miller]\label{Miller}
For every $\alpha\in\R$,
$\alpha$ is strongly Chaitin random if and only if
$\alpha$ is $2$-random.
\qed
\end{theorem}

The implication (i) $\Rightarrow$ (ii) of Theorem~\ref{main3}
is then proved as follows,
based on Lemma~\ref{H-unboundedf}, Theorem~\ref{Miller},
and the fact that $\Omega_V$ is an r.e.~real.

\begin{proof}[Proof of (i) $\Rightarrow$ (ii) of Theorem~\ref{main3}]
Let $V$ and $W$ be optimal computers.
For an arbitrary total recursive function $f\colon \N^+\to \N$,
assume that
there exist an oracle deterministic Turing machine $M$
and $c\in\N$ such that,
for all $n\in\N^+$,
$M^{\{\reste{\Omega_V}{n}\}}(n)=\rest{\Dom W}{n+f(n)-c}$.
Then, by considering the following procedure,
we first see that $n+f(n)<H(\rest{\Omega_V}{n})+O(1)$
for all $n\in\N^+$.

Given $\rest{\Omega_V}{n}$,
one first calculates the finite set $\rest{\Dom W}{n+f(n)-c}$
by simulating the computation of $M$
with the input $n$ and the oracle $\rest{\Omega_V}{n}$.
Then,
by calculating the set $\{\,W(p)\mid p\in\rest{\Dom W}{n+f(n)-c}\}$
and picking any one finite binary string $s$ which is not in this set,
one can obtain $s\in\X$ such that $n+f(n)-c<H_W(s)$.

Thus, there exists a partial recursive function
$\Psi\colon \X\to \X$ such that,
for all $n\in\N^+$, $n+f(n)-c<H_W(\Psi(\rest{\Omega_V}{n}))$.
It follows from the optimality of $W$ and \eqref{Psi} that
\begin{equation}\label{n+f(n)-lower-bound}
  n+f(n)<H(\rest{\Omega_V}{n})+O(1)
\end{equation}
for all $n\in\N^+$.

Now, let us assume contrarily that
the function $f$ is not bounded to the above.
Then it follows from Lemma~\ref{H-unboundedf} that
$H(n)\le f(n)$ for infinitely many $n\in\N^+$.
Combining this with \eqref{n+f(n)-lower-bound}
we see that $\Omega_V$ is strongly Chaitin random.
Thus, by Theorem~\ref{Miller},
$\Omega_V$ is $2$-random
and therefore
there exist an optimal oracle computer $R$ and $d\in\N$
such that
\begin{equation}\label{omega-2-random}
  n-d\le H_R^{\Dom U}(\rest{\Omega_V}{n})
\end{equation}
for all $n\in\N^+$.

On the other hand,
$\Omega_V\le_T \Dom U$ holds,
as shown in Theorem~\ref{main1} in a stronger form.
Thus, using \eqref{oracle-minimal}
we can show that
$H_R^{\Dom U}(\rest{\Omega_V}{n})\le 2\log_2 n+O(1)$
for all $n\in\N^+$.
However,
this contradicts \eqref{omega-2-random},
and the proof is completed.
\end{proof}

On the other hand,
in order to prove
the implication (ii) $\Rightarrow$ (i) of Theorem~\ref{main3},
we need Theorem~\ref{OmegaVhaltC} below.
Theorem~\ref{OmegaVhaltC} can be proved based on
Fact~\ref{Chaitin} and Corollary~\ref{oc-c}.
For completeness,
however,
we include
in Appendix~\ref{proof-OmegaVhaltC}
a direct and self-contained proof of Theorem~\ref{OmegaVhaltC}
without using Corollary~\ref{oc-c}.

\begin{theorem}\label{OmegaVhaltC}
Let $V$ be an optimal computer,
and let $C$ be a computer.
Then there exist an oracle deterministic Turing machine $M$
and $d\in\N$ such that,
for all $n\in\N^+$,
$M^{\{\reste{\Omega_V}{n+d}\}}(n)=\rest{\Dom C}{n}$,
where the finite subset $\rest{\Dom C}{n}$ of $\X$ is
represented as a finite binary string in a certain format.
\qed
\end{theorem}

Then the proof of the implication (ii) $\Rightarrow$ (i) of
Theorem~\ref{main3} is as follows.

\begin{proof}[Proof of (ii) $\Rightarrow$ (i) of Theorem~\ref{main3}]
Let $V$ and $W$ be optimal computers.
For an arbitrary total recursive function $f\colon \N^+\to \N$,
assume that the function $f$ is bounded to the above.
Then there exists $d_1\in\N$ such that $f(n)\le d_1$ for all $n\in\N^+$.
On the other hand,
by Theorem~\ref{OmegaVhaltC},
there exist an oracle deterministic Turing machine $M$
and $d_2\in\N$ such that,
for all $n\in\N^+$,
$M^{\{\reste{\Omega_V}{n+d_2}\}}(n)=\rest{\Dom W}{n}$,
where the finite subset $\rest{\Dom W}{n}$ of $\X$ is
represented as a finite binary string in a certain format.
We set $c=d_1+d_2$.
Then, by the following procedure,
we see that
there exists an oracle deterministic Turing machine $M$ such that,
for all $n\in\N^+$,
$M^{\{\reste{\Omega_V}{n}\}}(n)=\rest{\Dom W}{n+f(n)-c}$.

Given $n$ and $\rest{\Omega_V}{n}$ with $n>d_2$,
one first calculates the finite set $\rest{\Dom W}{n-d_2}$
by simulating the computation of $M$
with the input $n-d_2$ and the oracle $\{\rest{\Omega_V}{n}\}$.
One then calculates and outputs $\rest{\Dom W}{n+f(n)-c}$.
This is possible since $n+f(n)-c\le n+d_1-c=n-d_2$.
\end{proof}

Note that, as a variant of Theorem~\ref{main3},
we can prove the following theorem as well,
in a similar manner to the proof of Theorem~\ref{main3}.

\begin{theorem}[variant of the main result III]\label{variant-main3}
Let $V$ and $W$ be optimal computers,
and let $f\colon \N^+\to \N$ be a total recursive function.
Then the following two conditions are equivalent:
\begin{enumerate}
\item There exist an oracle deterministic Turing machine $M$
  and $c\in\N$ such that,
  for all $n\in\N^+$,
  $M^{\left\{\reste{\Omega_V}{n-f(n)+c}\right\}}(n)=\rest{\Dom W}{n}$,
  where the finite subset $\rest{\Dom W}{n}$ of $\X$ is
  represented as a finite binary string in a certain format.
\item The function $f$ is bounded to the above.\qed
\end{enumerate}
\end{theorem}

For completeness,
we give a proof of the implication (i) $\Rightarrow$ (ii),
i.e., the difficult part,
of Theorem~\ref{variant-main3} as follows.

\begin{proof}[Proof of (i) $\Rightarrow$ (ii) of Theorem~\ref{variant-main3}]
Let $V$ and $W$ be optimal computers.
For an arbitrary total recursive function $f\colon \N^+\to \N$,
assume that
there exist an oracle deterministic Turing machine $M$
and $c\in\N$ such that,
for all $n\in\N^+$,
$M^{\left\{\reste{\Omega_V}{n-f(n)+c}\right\}}(n)=\rest{\Dom W}{n}$.
Then, by considering the following procedure,
we first see that $n<H(n,\rest{\Omega_V}{n-f(n)+c})+O(1)$
for all $n\in\N^+$.

Given $n$ and $\rest{\Omega_V}{n-f(n)+c}$,
one first calculates the finite set $\rest{\Dom W}{n}$
by simulating the computation of $M$
with the input $n$ and the oracle $\left\{\rest{\Omega_V}{n-f(n)+c}\right\}$.
Then,
by calculating the set $\{\,W(p)\mid p\in\rest{\Dom W}{n}\}$
and picking any one finite binary string $s$ which is not in this set,
one can obtain $s\in\X$ such that $n<H_W(s)$.

Thus, there exists a partial recursive function
$\Psi\colon \N^+\times\X\to \X$ such that,
for all $n\in\N^+$, $n<H_W(\Psi(n,\rest{\Omega_V}{n-f(n)+c}))$.
It follows from the optimality of $W$ and \eqref{Psi} that
\begin{equation}\label{n-f(n)-lower-bound}
  n<H(n,\rest{\Omega_V}{n-f(n)+c})+O(1)
\end{equation}
for all $n\in\N^+$.

Now, let us assume contrarily that
the function $f$ is not bounded to the above.
It is then easy to show that
there exists an increasing total recursive function
$g\colon \N^+\to \N^+$
such that the function $f(g(k))$ of $k$ is increasing.
Note that $H(s,t)\le H(s)+H(t)+O(1)$ holds for all $s,t\in\X$
by \eqref{minimal}.
It follows from \eqref{n-f(n)-lower-bound} that
\begin{equation}\label{Hfg}
  g(k)-H(g(k))<H(\rest{\Omega_V}{g(k)-f(g(k))+c})+O(1)
\end{equation}
for all $k\in\N^+$.
On the other hand,
note that $\lim_{n\to\infty}n-H(n)=\infty$ holds by \eqref{eq: fas}.
Hence
we have $\lim_{k\to\infty}g(k)-H(g(k))=\infty$
since the function $g$ is increasing.
Based on \eqref{Hfg},
it is then easy to see that
the function $g(k)-f(g(k))+c$ of $k$ is not bounded to the above.
Therefore
there exists an increasing total recursive function
$h\colon \N^+\to \N^+$
such that $g(h(l))-f(g(h(l)))+c\ge 1$ for all $l\in\N^+$
and the function $g(h(l))-f(g(h(l)))+c$ of $l$ is increasing.
For clarity,
we define a total recursive function $m\colon \N^+\to \N^+$
by $m(l)=g(h(l))-f(g(h(l)))+c$.
Since $m$ is an increasing function,
it is then easy to see that
there exists a partial recursive function
$\Phi\colon \N^+\to \N^+$ such that
$\Phi(m(l))=g(h(l))$ for all $l\in\N^+$.
Thus, based on \eqref{Psi}, it is shown that
\begin{equation*}
  H(g(h(l)),\rest{\Omega_V}{m(l)})
  \le
  H(\rest{\Omega_V}{m(l)})+O(1)
\end{equation*}
for all $l\in\N^+$.
It follows from \eqref{n-f(n)-lower-bound} that
\begin{equation}\label{mfHm}
  m(l)+f(g(h(l)))<H(\rest{\Omega_V}{m(l)})+O(1)
\end{equation}
for all $l\in\N^+$.
On the other hand,
note that
the total recursive function $f(g(h(l)))$ of $l$ is increasing
and therefore not bounded to the above.
Thus, in a similar manner to the proof of Lemma~\ref{H-unboundedf}
we can show that
$H(m(l))\le f(g(h(l)))$ for infinitely many $l\in\N^+$.
It follows from \eqref{mfHm} that
\begin{equation*}
  m(l)+H(m(l))<H(\rest{\Omega_V}{m(l)})+O(1)
\end{equation*}
for infinitely many $l\in\N^+$.
Since the function $m$ is increasing,
we see that $\Omega_V$ is strongly Chaitin random.

Hereafter,
in the same manner as the proof of
the implication (i) $\Rightarrow$ (ii) of Theorem~\ref{main3},
we can derive a contradiction using Theorem~\ref{Miller}.
This completes the proof.
\end{proof}

\section*{Acknowledgments}

This work was supported
by KAKENHI, Grant-in-Aid for Scientific Research (C) (20540134),
by SCOPE
from the Ministry of Internal Affairs and Communications of Japan,
and by CREST from Japan Science and Technology Agency.


\appendix

\section{The proof of Theorem~\ref{time}}
\label{proof-time}

We here prove Theorem~\ref{time}
using Lemma~\ref{longer-running-time} below.
Let $V$ be an optimal computer,
and let $M$ be a deterministic Turing machine which computes $V$.

\begin{lemma}\label{longer-running-time}
There exists $d\in\N$ such that,
for every $p\in\Dom V$,
there exists $q\in\Dom V$
for which $\abs{q}\le \abs{p}+d$ and
the running time of $M$ on the input $q$ is longer than
the running time of $M$ on the input $p$.
\end{lemma}

\begin{proof}
Consider the computer $C$ such that
(i) $\Dom C=\Dom V$ and
(ii) for every $p\in\Dom V$,
$C(p)=1^{2\abs{p}+T(p)+1}$,
where $T(p)$ is the running time of $M$ on the input $p$.
It is easy to see that such a computer $C$ exists.
Then,
since $V$ is an optimal computer,
from the definition of an optimal computer
there exists $d_1\in\N$ with the following property;
if $p\in\Dom C$, then there is $q$ for which
$V(q)=C(p)$ and $\abs{q}\le\abs{p}+d_1$.

Thus,
for each $p\in\Dom V$ with $\abs{p}\ge d_1$,
there is $q$ for which
$V(q)=C(p)$ and $\abs{q}\le\abs{p}+d_1$.
It follows that
\begin{equation}\label{output-input}
  \abs{V(q)}=2\abs{p}+T(p)+1>\abs{p}+d_1+T(p)\ge\abs{q}+T(p).
\end{equation}
Note that
exactly $\abs{q}$ cells on the tapes of $M$ have the symbols $0$ or $1$
in the initial configuration of $M$ with the input $q$,
while at least $\abs{V(q)}$ cells on the tape of $M$,
on which the output is put, have the symbols $0$ or $1$
in the resulting final configuration of $M$.
Since $M$ can write at most one $0$ or $1$ on the tape,
on which an output is put, every one step of its computation,
the running time $T(q)$ of $M$ on the input $q$
is bounded to the below by the difference $\abs{V(q)}-\abs{q}$.
Thus, by \eqref{output-input}, we have $T(q)>T(p)$.

On the other hand,
since $\Dom V$ is not a recursive set,
the function $T^M_n$ of $n\ge L_M$ is not bounded to the above.
Therefore, there exists $r_0\in\Dom V$ such that,
for every $p\in\Dom C$ with $\abs{p}<d_1$, $T(r_0)>T(p)$.
By setting $d_2=\abs{r_0}$
we then see that, for every $p\in\Dom C$ with $\abs{p}<d_1$,
$\abs{r_0}\le\abs{p}+d_2$.

Thus,
by setting $d=\max\{d_1,d_2\}$
we see that, for every $p\in\Dom V$,
there is $q\in\Dom V$ for which
$\abs{q}\le\abs{p}+d$ and $T(q)>T(p)$.
This completes the proof.
\end{proof}

The proof of Theorem~\ref{time} is given as follows.

\begin{proof}[Proof of Theorem~\ref{time}]
By considering the following procedure,
we first show that $n\le H(T^M_n,n)+O(1)$ for all $n\ge L_M$.

Given $(T^M_n,n)$ with $n\ge L_M$,
one first calculates the finite set $\rest{\Dom V}{n}$
by simulating the computation of $M$ with the input $p$
until at most $T^M_n$ steps,
for each $p\in\X$ with $\abs{p}\le n$.
Then,
by calculating the set $\{\,V(p)\mid p\in\rest{\Dom V}{n}\}$
and picking any one finite binary string $s$ which is not in this set,
one can obtain $s\in\X$ such that $n<H_V(s)$.

Hence, there exists a partial recursive function
$\Psi\colon \N\times\N^+\to \X$ such that,
for all $n\ge L_M$,
$n<H_V(\Psi(T^M_n,n))$.
It follows from the optimality of $V$ and \eqref{Psi} that
$n<H(T^M_n,n)+O(1)$
for all $n\ge L_M$.

For each $p\in\Dom V$,
let $T(p)$ be the running time of $M$ on the input $p$.
It follows from Lemma~\ref{longer-running-time} that
there exists $d\in\N$ such that, for every $p\in\Dom V$,
there exists $q\in\Dom V$
for which $\abs{q}\le \abs{p}+d$ and $T(q)>T(p)$. 
By considering the following procedure,
we next show that $H(T^M_n,n)\le H(T^M_n)+O(1)$ for all $n\ge L_M$.

Given $T^M_n$ with $n\ge L_M$,
one first simulates the computation of $M$ with the input $p$
until at most $T^M_n$ steps,
one by one for each element $p$ in $\X$,
in the order defined on the ordered set $\X$.
Due to the definition of $T^M_n$,
during the simulations
one can eventually find the first element $p_0$ of $\X$ such that
$T(p_0)=T^M_n$.
For this $p_0$,
$\abs{p_0}\le n$ due to the definition of $T^M_n$, 
and there exists $q\in\Dom V$
for which $\abs{q}\le \abs{p_0}+d$ and $T(q)>T^M_n$.
For this $q$, $\abs{q}>n$ due to the definition of $T^M_n$ again.
Therefore $d\ge 1$ and $\abs{p_0}\le n<\abs{p_0}+d$.
Thus, there are still only $d$ possibilities of $n$,
so that one needs only $\lceil \log_2 d\rceil$ bits more
in order to determine $n$.

Thus, there exists a partial recursive function
$\Phi\colon \N\times\X\to\N\times \N^+$ such that,
for every $n\ge L_M$,
there exists $s\in\X$ with the properties that
$\abs{s}=\lceil \log_2 d\rceil$ and $\Phi(T^M_n,s)=(T^M_n,n)$.
It follows from \eqref{Psi} and \eqref{minimal} that
$H(T^M_n,n)\le
H(T^M_n)+
\max\{H(s)\mid s\in\X\ \&\ \abs{s}=\lceil \log_2 d\rceil\}+O(1)$
for all $n\ge L_M$.

Finally, we show that $H(T^M_n)\le n+O(1)$ for all $n\ge L_M$.
Let us consider the computer $C$ such that
(i) $\Dom C=\Dom V$ and
(ii) for every $p\in\Dom V$, $C(p)=T(p)$.
Obviously, such a computer $C$ exists.
Then, by \eqref{minimal} we see that,
for every $p\in\Dom V$, $H(T(p))\le\abs{p}+O(1)$.
For each $n\ge L_M$,
it follows from the definition of $T^M_n$
that there exists $r\in\Dom V$ such that
$\abs{r}\le n$ and $T(r)=T^M_n$.
Hence,
$H(T^M_n)=H(T(r))\le\abs{r}+O(1)\le n+O(1)$.
This completes the proof.
\end{proof}

\section{The proof of Lemma~\ref{H-unboundedf}}
\label{proof-H-unboundedf}

Lemma~\ref{H-unboundedf} is proved as follows.

\begin{proof}[Proof of Lemma~\ref{H-unboundedf}]
Contrarily, assume that
there exists $c\in\N^+$ such that, for every $n\ge c$, $f(n)<H(n)$.
Then, since $f$ is not bounded to the above,
it is easy to see that
there exists a total recursive function $\Psi\colon \N^+\to \N^+$
such that, for every $k\in\N^+$, $k<H(\Psi(k))$.
Thus, using \eqref{Psi} we see that $k<H(k)+O(1)$ for all $k\in\N^+$. 
On the other hand,
using \eqref{eq: fas} we have
$H(k)\le 2\log_2 k+O(1)$ for all $k\in\N^+$.
Therefore $k<2\log_2 k+O(1)$ for all $k\in\N^+$.
However,
we have a contradiction on letting $k\to\infty$ in this inequality,
and the result follows.
\end{proof}

\section{The proof of Theorem~\ref{OmegaVhaltC}}
\label{proof-OmegaVhaltC}

In what follows,
we prove Theorem~\ref{OmegaVhaltC}
in a direct manner without using Corollary~\ref{oc-c}.

\begin{proof}[Proof of Theorem~\ref{OmegaVhaltC}]
In the case where $\Dom C$ is a finite set,
the result is obvious.
Thus,
in what follows,
we assume that $\Dom C$ is an infinite set.

Let $p_0,p_1,p_2,p_3, \dotsc$ be
a particular recursive enumeration of $\Dom C$,
and let $D$ be a computer such that
$\Dom D=\Dom C$ and $D(p_i)=i$ for all $i\in\N$. 
Recall here that we identify $\X$ with $\N$.
It is also easy to see that such a computer $D$ exists.
Since $V$ is an optimal computer,
from the definition of optimality of a computer
there exists $d\in\N$ such that, for every $i\in\N$,
there exists $q\in\X$ for which $V(q)=i$ and $\abs{q}\le \abs{p_i}+d$.
Thus,
$H_V(i)\le\abs{p_i}+d$ for every $i\in\N$.
For each $s\in \X$,
we define $P_V(s)$ as $\sum_{V(p)=s}2^{-\abs{p}}$.
Then, for each $i\in\N$,
\begin{equation}\label{imp}
  P_V(i)\ge 2^{-H_V(i)}\ge 2^{-\abs{p_i}-d}.
\end{equation}
Then, by the following procedure,
we see that
there exists an oracle deterministic Turing machine $M$
such that,
for all $n\in\N^+$,
$M^{\{\reste{\Omega_V}{n+d}\}}(n)=\rest{\Dom C}{n}$.

Given $n$ and $\rest{\Omega_V}{n+d}$,
one can find $k_e\in\N$ such that
$\sum_{i=0}^{k_e} P_V(i)>0.(\rest{\Omega_V}{n+d})$.
This is possible because
$0.(\rest{\Omega_V}{n+d})<\Omega_V$ and
$\lim_{k\to\infty}\sum_{i=0}^k P_V(i)=\Omega_V$.
It follows that
\begin{equation*}
  \sum_{i=k_e+1}^{\infty} P_V(i)
  =\Omega_V-\sum_{i=0}^{k_e} P_V(i)
  <\Omega_V-0.(\rest{\Omega_V}{n+d})<2^{-n-d}.
\end{equation*}
Therefore, by \eqref{imp},
\begin{equation*}
  \sum_{i=k_e+1}^{\infty} 2^{-\abs{p_i}}
  \le 2^{d}\sum_{i=k_e+1}^{\infty} P_V(i)
  <2^{-n}.
\end{equation*}
It follows that,
for every $i>k_e$,
$ 2^{-\abs{p_i}}<2^{-n}$
and therefore
$n<\abs{p_i}$.
Hence,
\begin{equation*}
  \rest{\Dom C}{n}
  =\{\,p\in\Dom C\mid \abs{p}\le n\,\}
  =\{\,p_i\mid i\le k_e\;\&\;\abs{p_i}\le n\,\}.
\end{equation*}
Thus, by calculating the finite set
$\{\,p_i\mid i\le k_e\;\&\;\abs{p_i}\le n\,\}$,
one can obtain the set $\rest{\Dom C}{n}$.
\end{proof}

\end{document}